\documentclass[12pt,a4paper]{amsart}
\usepackage{a4wide}
\usepackage{amsfonts,amsthm,amsmath,amssymb,amscd,color}
\usepackage{graphicx}

\theoremstyle{plain}
\newtheorem{lem}{Lemma}
\newtheorem{prop}[lem]{Proposition}
\newtheorem{thm}[lem]{Theorem}
\newtheorem*{thm*}{Theorem}

\newtheorem*{cor*}{Corollary}

\newtheorem*{prop*}{Proposition}

\theoremstyle{definition}
\newtheorem{defn}[lem]{Definition}
\newtheorem*{defn*}{Definition}

\newtheorem*{ex*}{Example}
\newtheorem{rem}[lem]{Remark}
\newtheorem*{rem*}{Remark}

\theoremstyle{remark}

\DeclareMathOperator{\cl}{cl}

\DeclareMathOperator{\Lip}{Lip}

\DeclareMathOperator{\supp}{supp}

\renewenvironment{proof}{\medbreak{\noindent\em Proof }}{~{\hfill$\bullet$\bigbreak}}

\def\d{{\rm d}}

\begin{document}

\title{The Central Limit Theorem for Function Systems on the Circle}
\author{Tomasz Szarek and Anna Zdunik}
\address{Tomasz Szarek, Institute of Mathematics, University of Gda\'nsk, Wita Stwosza 57, 80-952 Gda\'nsk, Poland}
\email{szarek@intertele.pl}
\address{Anna Zdunik, Institute of Mathematics, University of Warsaw,
ul.~Banacha~2, 02-097 Warszawa, Poland}
\email{A.Zdunik@mimuw.edu.pl}



\thanks{The research partially supported by the
  Polish NCN grants 2016/21/B/ST1/00033 (Tomasz Szarek) and 2014/13/B/ST1/04551 (Anna Zdunik).}
\begin{abstract} The Central Limit Theorem for iterated functions systems on the circle is proved. We study also ergodicity of such systems.
\end{abstract}

\maketitle

\section{Introduction}
In this paper we deal with an iterated function systems (ifs -- for short)  generated by finite families of homeomorphisms of the circle.

Our main goals are the following: first,  to prove the Central Limt Theorem {\bf (CLT)} for Lipschitz continuous observables, and the Markov process generated by the  ifs. This is done under natural mild assumptions, i.e, minimality of the action of the corresponding semigroup on the circle. No additional regularity of the maps is required. In this way, we answer the question which was left open in our previous paper \cite{szarek_zdunik}. The proof is based on the result due to Maxwell and Woodroofe (see \cite{MW}) which provides a sufficient condition for the Central Limit Theorem for an arbitrary stationary Markov chain. It is worth mentioning here that our considerations allow us to show the CLT for ifs's starting at an arbitrary initial distribution. Similar result has been obtained recently by Komorowski and Walczuk (see \cite{KW}) but developed techniques allow them to consider only Markov chains satisfying spectral gap property in the Wasserstein metric.

Our second purpose  is to provide some insights into Markov operators with the e-property.
The e--property is  a very useful tool in studying ergodic properties of Markov operators and semigroups of Markov operators. It was introduced to deal with stochastic partial differential equations in infinite dimensional Hilbert spaces (see for instance \cite{KPS}) but it is also very helpful while studing ifs's. 

In Sections \ref{section:e_property}  and ~\ref{section:ergodicity}  we show how this property can be easily verified and then used to provide alternative  proofs of some known results: ergodicity and asymptotic stability  of the iterated function systems, again, under natural mild assumptions.

\section{Notation and basic information about Markov operators.}

Since we shall deal with systems on the circle, we restrict this short presentation to the case of compact metric spaces. The general theory is developed for Polish spaces.

Let $(S, d)$ be a compact metric  space. By $\mathcal M_1(S)$  we denote the set of all  probability measures  on  the $\sigma$--algebra of Borel sets $\mathcal B(S)$. 
By $C(S)$ we denote the family of all continuous functions equipped with the supremum norm $\|\cdot\|$ and by $\Lip (S)$ we denote the family of all Lipschitz functions. For $f\in \Lip (S)$ by $\Lip f$ we denote its Lipschitz constant. For brevity we shall use the notion of scalar product:
$$
\langle f, \mu\rangle :=\int_S f(x)\mu (\d x)
$$
for any bounded Borel measurable function $f: S\to\mathbb R$ and $\mu\in\mathcal M_1 (S)$.




An operator $P: \mathcal M (S)\to \mathcal M (S)$ is called a {\it Markov operator} if it satisfies the following two conditions:
\begin{itemize}
\item[{\bf 1)}] positive linearity: $P(\lambda_1\mu_1+\lambda_2\mu_2)=\lambda_1 P\mu_1+\lambda_2 P\mu_2$ for $\lambda_1, \lambda_2\ge 0$, $\mu_1, \mu_2\in\mathcal M (S)$;
\item[{\bf 2)}] preservation of the norm: $P\mu (S)=\mu (S)$ for $\mu\in\mathcal M (S)$.
\end{itemize}

A Markov operator $P$ is called a {\it Feller operator} if there is a linear operator $U: C(S)\to C(S)$  such that
$$
\int_{S} Uf(x)\mu (\d x)=\int_{S} f(x) P\mu(\d x)\quad\text{for $f\in C(S)$, $\mu\in\mathcal M $.}
$$

A Markov operator $P: \mathcal M (S)\to \mathcal M (S)$ is called {\it nonexpansive} (with respect to the Wasserstein metric) if
$$
W_1(P\mu, P\nu)\le W_1(\mu, \nu)\qquad\text{for any $\mu, \nu\in\mathcal M_{1} (S)$.}
$$

A measure $\mu_*$ is called {\it invariant} if $P\mu_*=\mu_*$. An operator $P$ is called {\it asymptotically stable} if it has a unique invariant measure $\mu_*\in\mathcal M_1(S)$ such that
the sequence $(P^n\mu)_{n\ge 1}$ converges in the $*$--weak topology to $\mu_*$ for any $\mu\in\mathcal M_1(S)$, i.e.,
$$
\lim_{n\to\infty}\int_{S} f(x) P^n\mu (\d x)=\int_{S} f(x)\mu_*(\d x)
$$
for any $f\in C(S)$.

For any Markov operator $P$ we define the the multifunction $\mathcal P:  2^S\to 2^S$ by the formula
$$
\mathcal P(A)=\bigcup_{x\in A}\supp P\delta_x\qquad\text{for $A\subset S$}.
$$
\section{E-property}\label{section:e_property}
The e--property seems to be a very useful tool in studying ergodic properties of Markov operators and semigroups of Markov operators on Polish spaces.
Following \cite{KPS}, we say that a Feller operator $P$ satisfies the {\it e--property} if for any $x\in S$ and a Lipschitz  function $f\in C(S)$ we have 
$$
\lim_{y\to x} \sup_{n\in\mathbb N}|U^n f (y) - U^n f (x)|=0,
$$
i.e., if the family of iterates $\{U^n f: n\in\mathbb N\}$ is equicontinuous.

\begin{prop}\label{prop1} Let $P$ be a Feller operator. If $P$ satisfies the e--property, then 
$$
\supp \mu\cap\supp\nu=\emptyset
$$
for any different ergodic invariant measures $\mu, \nu\in {\mathcal M}(S)$.
\end{prop}

\begin{proof} The proof may be derived from \cite{HHSz} (see also \cite{KSzS, KPS}). Indeed, in Lemma 3.4 we proved that if $x\in\supp\mu$, where  $\mu\in\mathcal M_1(S)$ is an ergodic invariant measure, then the sequence
$(n^{-1}\sum_{i=1}^n P^n\delta_x)_{n\ge 1}$ converges in the $*$--weak topology to $\mu$.
Hence our assertion follows immediately.
\end{proof}

D. Worm slightly generalized the e--property introducing the Ces\'aro e--property (see  \cite{Worm-thesis:2010}). Namely,  a Feller operator $P$ will satisfy the {\it Ces\'aro e--property} at $x\in S$ if for any Lipschitz  function $f\in C(S)$ we have 
$$
\lim_{y\to x} \sup_{n\in\mathbb N}\left|\frac{1}{n}\sum_{k=1}^n U^k f (y) - \frac{1}{n}\sum_{k=1}^n U^k f (x)\right|=0.
$$

For Feller operators with the Ces\'aro e--property  the following proposition holds. Its    proof is the same as the proof of  Proposition 2 in \cite{KSzS}.

\begin{prop}\label{prop1} Let $(S, d)$ be a compact metric space and let $P$ be a Feller operator. Assume that there exists an open subset $S_0\subset S$ such that $\mathcal P(S_0)\subset S_0$ and $\mu (S_0)=1$ for any invariant measure $\mu\in\mathcal M_1(S)$.  If $P$ satisfies the Ces\'aro e--property at any point $x\in S_0$, then for any ergodic invariant measure $\mu_*\in\mathcal M_1(S)$ and every $x\in S_0\cap\supp\mu_*$ the sequence $(\frac{1}{n}\sum_{k=1}^n P^n\delta_x)_{n\ge 1}$ converges weakly to $\mu_*$.
\end{prop}

\section{Ergodicity for iterated function systems on the circle}\label{section:ergodicity}

Iteration of homeomorphisms on the circle has been widely studied recently. For further references see \cite{Navas, ghys, navas2, szarek_zdunik} and the references therein. The main purpose of this section is to prove that Markov operators corresponding to iterated function systems on the circle have strong metric properties, i.e. nonexpansiveness, the e--property and Ces\'aro e-property. These properties imply straightforwardly the ergodic properties of the systems. In this way we may easily derive ergodicity under the most general condition on the system (see \cite{malicet}).

Let $\mathbf{S}^1$ denote the circle with the counterclockwise orientation. We will denote by $[x, y]$ the closed interval form $x$ to $y$ according to this orientation. The distance between $x, y\in \mathbf{S}^1$ is the shorter of the lengths of the intervals $[x, y]$ and $[y, x]$.  We will denote this distance by $d (x, y)$. 

By $H^+$ we shall denote the set of all orientation preserving circle homeomorphisms.
Let $\Gamma=\{g_1,\dots,g_k\}\subset H^+$ be a finite collection of homeomorphisms. Put $\Sigma_n=\{1,\dots, k\}^n$,  and let $\Sigma_*=\bigcup_{n=1}^\infty\Sigma_n$ be the collection of all finite words with entries from $\{1,\dots ,k\}$.
For a sequence ${\bf i}\in\Sigma_*$, ${\bf i}=(i_1,\dots ,i_n)$, we denote by $|{\bf i}|$ its length (equal to $n$).  We denote by $\Sigma$ the product space  $\{1,\dots, k\}^{\mathbb N}$.

 We consider the action of the semigroup generated by  $\Gamma$, i.e., the action of all compositions $g_{\bf i}=g_{i_n, i_{n-1},\ldots, i_1}=g_{i_n}\circ g_{i_{n-1}}\circ\cdots\circ g_{i_1}$, where ${\bf i}=(i_1,\dots i_n)\in\Sigma_*$.   
 \begin{defn}
 The orbit of a point $x\in \mathbf{S}^1$ is the set 
 $$
 \mathcal O(x)=\{g_{\bf i}(x): {\bf i}\in \Sigma_*\}.
 $$

 

 In the case when all the orbits are dense the action of $\Gamma$ is called {\it minimal}. Equivalently, the action of $\Gamma$ is minimal if for every $\Gamma$--invariant closed subset $A\subset \mathbf{S}^1$ either $A=\emptyset$ or $A=\mathbf{S}^1$.
\end{defn} 
 
  \vskip3mm
 
 Let ${\bf p}= (p_1,\dots p_k)$ be a probability distribution on $\{1,\dots,k\}$.  We denote by $\mathbb P$ the product probability distribution on $\Sigma$.  Clearly, $\bf p$  defines a probability distrubution $p$ on   $\Gamma$, by putting $p(g_j)=p_j$. We assume that all $p_i$'s are strictly positive. The pair $(\Gamma, {\bf p})$ will be called an {\it iterated function system}.
 
The Markov operator $P: \mathcal M (\mathbf{S}^1)\to\mathcal M (\mathbf{S}^1)$ of the form
$$
P\mu=\sum_{g\in\Gamma} p(g)\mu\circ g^{-1},
$$
where $\mu\circ g^{-1}(A)=\mu (g^{-1}(A))$ for $A\in\mathcal B(\mathbf S^1)$, describes the evolution of distribution due to action of randomly chosen homeomorphisms from the collection $\Gamma$. 
It is a Feller operator, i.e., the operator $U: C(\mathbf S^1)\to C(\mathbf S^1)$ given by the formula
$$
U f(x)=\sum_{g\in\Gamma} p(g) f(g(x))\qquad\text{for $f\in C(\mathbf S^1)$ and $x\in\mathbf S^1$}
$$
is its dual.
We shall illustrate   usefulness of the notion of the e-property, providing a very simple proof of the following:
\begin{prop} Let $\Gamma^{-1}=\{g_1^{-1},\ldots, g_k^{-1}\}$ act minimally and  let ${\bf p}=(p_1,\dots p_k)$ be a probability distribution on $\{1,\dots,k\}$. Then the operator $P$ corresponding to the iterated function system $(\Gamma, {\bf p})$ satisfies the e--property.
Moreover, $P$ admits a unique invariant measure. 
\end{prop}

\begin{proof} Let $\tilde\mu\in\mathcal M_1(\mathbf S^1)$ be an arbitrary invariant measure for the iterated function system
$(\Gamma^{-1}, p)$. Since $\Gamma^{-1}$ acts minimally, the support of $\tilde\mu$ equals ${\mathbf S}^{1}$. We easily check that $\tilde\mu$ has no atoms. To do this take the atom $u$ with maximal measure. From the fact that $\tilde\mu$ is invariant for $P$ we obtain that
$F=\{v\in {\mathbf S}^{1}: \tilde\mu(\{v\})=\tilde\mu(\{u\})\}$ is invariant for $\Gamma$ and consequently it is also invariant for $\Gamma^{-1}$, i.e., $g_i(F)=F$ and $g_i^{-1}(F)=F$ for $i=1,\ldots, k$. This contradicts the assumtion that $\Gamma^{-1}$ acts minimally. Indeed, from the fact that
$$
\tilde\mu(\{v\})=\sum_{i=1}^k p_i\tilde\mu(\{g_i(v)\}),
$$
we obtain that $\tilde\mu(\{g_i(v)\})=\tilde\mu(\{v\})$ for all $i=1,\ldots, k$. Since $g_i$ are homeomorphisms and the set $F$  is finite we obtain that $g_i(F)=F$ and $g_i^{-1}(F)=F$ for $i=1,\ldots, k$, which is impossible, since the action of $\Gamma^{-1}$ is minimal.

Define the function $\chi:{\mathbf S}^{1}\times {\mathbf S}^{1}\to\mathbb R_+$ by the formula
$$
\chi(x, y)=\min \left(\tilde\mu([x, y]), \tilde\mu([y, x])\right)\qquad\text{for $x, y\in{\mathbf S}^{1}$.}
$$
It is straigtforward to check that $\chi$ is a metric and convergence in $\chi$ is equivalent to the convergence in $d$.

Further, we may check that for any function $f: {\mathbf S}^{1}\to\mathbb R$ satisfying
$|f(x)-f(y)|\le \chi(x, y)$ for $x, y\in{\mathbf S}^{1}$ we have
$$
|Uf(x)-Uf(y)|\le \chi(x, y)\qquad\text{for $x, y\in{\mathbf S}^{1}$.}
$$
This follows from the definition of the operator $U$ and the fact that $|f(x)-f(y)|\le \tilde\mu([x, y])$ and
$|f(x)-f(y)|\le \tilde\mu([y, x])$. Indeed, we have
$$
|Uf(x)-Uf(y)|\le\sum_{g\in\Gamma} p(g)|f(g(x))-f(g(y)|\le \sum_{g\in\Gamma} p(g)\tilde\mu([g(x), g(y)]=\tilde\mu([x, y])
$$
and analogously
$$
|Uf(x)-Uf(y)|\le\sum_{g\in\Gamma} p(g)|f(g(x))-f(g(y)|\le \sum_{g\in\Gamma} p(g)\tilde\mu([g(y), g(x)]=\tilde\mu([y, x])
$$
and hence
$$
|Uf(x)-Uf(y)|\le\chi(x, y)
$$
for any $x, y\in{\mathbf S}^{1}$.
This finishes the proof of the e-property of the operator $P$. 

To complete the proof of our theorem we would like to apply Proposition \ref{prop1}. Therefore we have to check that $\supp\mu\cap\supp\nu\neq\emptyset$ for any $\mu, \nu\in\mathcal M_1({\mathbf S}^{1})$ invariant for $P$.  Assume, contrary to our claim, that $\supp\mu\cap\supp\nu=\emptyset$. Take the set $\Lambda$ of all intervals $I\subset {\mathbf S}^{1}\setminus (\supp\mu\cup\supp\nu)$ such that one of its ends belongs to $\supp\mu$ but the second to $\supp\nu$. Observe that $\tilde\mu (I)>0$ for all $I\in\Lambda$ and that, by compactness, there exists $I_0$ such that $\tilde\mu(I_0)=\inf_{I\in\Lambda}\tilde\mu (I)$. We easily see that $g(I_0)\in\Lambda$. Indeed, we have
\begin{equation}\label{e1_5.01.17}
\tilde\mu(I_0)=\sum_{g\in\Gamma} p(g)\tilde\mu (g(I_0))
\end{equation}
and since the interval $g(I_0)$ has the ends belonging to $\supp\mu$ and to $\supp\nu$, for $I_0$ had and $g(\supp\mu)\subset\supp\mu$ and $g(\supp\nu)\subset\supp\nu$ for all $g\in\Gamma$, we obtain that there is $J\in\Lambda$ such that $J\subset  g(I_0)$. Hence $\tilde\mu (g(I_0))\ge \tilde\mu(J)\ge\tilde\mu(I_0)$ and from equation (\ref{e1_5.01.17}) it follows that $\tilde\mu (g(I_0))=\tilde\mu(I_0)$ and consequently $g(I_0)\in\Lambda$. Otherwise there would be $h\in\Gamma$ and $J\subset h(I_0)$, $J\in\Lambda$ and $\tilde\mu(J)<\tilde\mu(h(I_0))\le\tilde\mu(I_0)$, by the fact that $\supp\tilde\mu={\mathbf S}^{1}$, contrary to the definition of $I_0$. Finally, observe that the set
$$
\mathcal H=\{J\in\Lambda: \tilde\mu(J)=\tilde\mu(I_0)\}
$$
 is finite and all its elements are disjoint open intervals. Further $g(\mathcal H)\subset\mathcal H$ and consequently
 $g(\mathcal H)=\mathcal H$ for any $g\in\Gamma$, by the fact that $g$ is a homeomorphism. 
 Consequently, for any $g$ and the finite set $F$ of all ends of the intervals $J$ from $\mathcal H$ we have
 $g(F)\subset F$ and therefore $g(F)=F$. Hence $g^{-1}(F)=F$ for any $g\in\Gamma$ and consequently $\Gamma^{-1}$ is not minimal, contrary to our assumption.
 \end{proof}
 
 The following theorem was proved in \cite{malicet}, with the proof involving a generalization of Lyapunov exponents. We want to provide a very simple argument, based only on the (independently proved) e-property.
\begin{thm} If $\Gamma=\{g_1,\ldots, g_k\}$ acts minimally, then for any probability vector ${\bf p}=(p_1,\dots p_k)$ the iterated function system $(\Gamma, {\bf p})$  admits a unique invariant measure. 
\end{thm}

\begin{proof} The iterated function system $(\Gamma^{-1}, {\bf p})$ satisfies the e--property. Denote by $\tilde P$ and $\tilde U$ the Markov operator and the dual operator corresponding to $(\Gamma^{-1}, {\bf p})$, respectively. From the proof of the previous proposition it follows that the hypothesis holds provided the unique invariant measure $\tilde\mu$ for $(\Gamma^{-1}, {\bf p})$ satisfies the condition $\supp\tilde\mu=\mathbf S^1$.

Now assume that $\mathbf S^1\setminus \supp\tilde\mu\neq\emptyset$. Let $(a, b)\subset\mathbf S^1\setminus \supp\tilde\mu$. 
Set 
$$
S_0=\bigcup_{n=1}^{\infty}\bigcup_{(i_1,\ldots, i_n)\in \Sigma_n} g_{i_1,\ldots, i_n}((a, b))
$$
and observe that $S_0$ is open and dense in $\mathbf S^1$, by the minimal action of $\Gamma$. Let $\mu_*\in\mathcal M_1(\mathbf S^1)$ be an ergodic invariant measure for $(\Gamma, p)$. Since $\supp\mu_*=\mathbf S^1$ and $g_i(S_0)\subset S_0$ for any $i=1, \ldots, k$ we have $\mu_*(S_0)>0$ and $U{\bf 1}_{S_0}={\bf 1}_{S_0}$. Thus $\mu_*(S_0)=1$, by ergodicity of $\mu_*$.
We are going to apply Proposition \ref{prop1} therefore we have to check that the Ces\'aro e-property holds at any $x\in S_0$. To do this fix $x\in S_0$ and $\varepsilon>0$.  Let $I\subset S_0$ be an open neighbourhood of $x$. Let $f: \mathbf S^1\to\mathbb R$ be a Lipschitz function with the Lipschitz constant $L$. Choose a finite set $\{x_0, \ldots, x_N\}\subset \mathbf S^1$ such that $|[x_{i-1}, x_i]|<\varepsilon/L$. Since $\frac{1}{n}\sum_{k=1}^n \tilde P^k\delta_x$ converges weakly to $\tilde\mu$ for any $x\in\mathbf {S}^1$ and $\tilde\mu(I)=0$ we have
$$
\frac{1}{n}\sum_{k=1}^n \tilde U^k{\bf 1}_I (x)\to 0\qquad\text{as $n\to\infty$.}
$$
On the other hand, we have
$$
\begin{aligned}
\frac{1}{n}\sum_{k=1}^n \tilde U^k{\bf 1}_I (x)&=\frac{1}{n}\sum_{k=1}^n\sum_{(i_1,\ldots, i_k)\in\Sigma_k} p_{i_1}\cdots p_{i_k} {\bf 1}_I (g_{i_1}^{-1}\circ\cdots\circ g_{i_k}^{-1}(x))\\
&=\frac{1}{n}\sum_{k=1}^n\sum_{(i_1,\ldots, i_k)\in\Sigma_k} p_{i_1}\cdots p_{i_k} {\bf 1}_{g_{i_1,\ldots, i_k} (I)} (x)\\
&=\frac{1}{n}\sum_{k=1}^n\sum_{(i_1,\ldots, i_k)\in\Sigma_k} p_{i_1}\cdots p_{i_k} {\bf 1}_{\{x\}} (g_{i_1,\ldots, i_k} (I)).
\end{aligned}
$$
Consequently, we have
$$
\frac{1}{n}\sum_{k=1}^n\sum_{(i_1,\ldots, i_k)\in\Sigma_k} p_{i_1}\cdots p_{i_k} {\bf 1}_{\{x_0,\ldots, x_M\}} (g_{i_1,\ldots, i_k} (I))\to 0\qquad\text{as $n\to\infty$.}
$$
Thus, for any $x, y\in I$, the interval $g_{i_1\cdots i_k}([x, y])$ typically will be located between some points $x_{i-1}$ and $x_i$, so that its length will be less than $\varepsilon/L$. Hence
$$
\begin{aligned}
&\limsup_{n\to\infty}\left|\frac{1}{n}\sum_{k=1}^n U^k f(x)-\frac{1}{n}\sum_{k=1}^n U^k f(y)\right|\\
&\le\limsup_{n\to\infty}\frac{1}{n}\sum_{k=1}^n\sum_{(i_1,\ldots, i_k)\in\Sigma_k} p_{i_1}\cdots p_{i_k} |f (g_{i_1}\circ\cdots\circ g_{i_k}(x))-f (g_{i_1}\circ\cdots\circ g_{i_k}(y))|\\
&\le \limsup_{n\to\infty} \frac{L}{n}\sum_{k=1}^n\sum_{(i_1,\ldots, i_k)\in\Sigma_k} p_{i_1}\cdots p_{i_k} |g_{i_1,\ldots, i_k}([x, y])|\le L\varepsilon/L=\varepsilon.
\end{aligned}
$$
Since $\varepsilon>0$ was arbitrary, the operator $P$ satisfies the Ces\'aro e--property. This completes the proof.
\end{proof}
 
\section{Central Limit Theorem}

Let $\Gamma=\{g_1,\ldots, g_k\}$ be a family of homeomorphisms on $\mathbf S^1$ and let ${\bf p}=(p_1,\dots p_k)$ be a probability vector. Let $\mu_*\in\mathcal M_1(\mathbf S^1)$ be an  invariant measure for the iterated function system $(\Gamma, {\bf p})$. By $(X_n)_{n\ge 0}$ we shall denote the stationary Markov chain corresponding to the iterated function system $(\Gamma, {\bf p})$. Let $\varphi: \mathbf S^1\to\mathbb R$ be a H\"older continuous function satisfying $\int_{\mathbf S^1} \varphi \d\mu_*=0$. Set 
$$
S_n:=
S_n(\varphi)=\varphi(X_0)+\ldots+\varphi(X_n)
$$
and
$$
S^*_n=\frac{1}{\sqrt{n}}S_n\qquad\text{for $n\ge 1$.}
$$
Our main purpose  in this section is to prove that $S_n^*$ is asymptotically normal (the {\bf CLT} theorem). Maxwell and Woodroofe in \cite{MW} studied  general Markov chains and formulated a simple sufficient condition for the {\bf CLT} which in our case takes the form
\begin{equation}\label{eq:1_6/03.17}
\sum_{n=1}^{\infty} n^{-3/2} \|\sum_{k=1}^n U^k\varphi\|_{L^2(\mu_*)}<\infty,
\end{equation}
where $\|\cdot\|_{L^2(\mu_*)}$ denotes the $L^2(\mu_*)$ norm.
More precisely, the result proved in \cite{MW} says that if  \eqref{eq:1_6/03.17} holds, then the limit
$\sigma^2=\lim_{n\to\infty} E(S_n^{*2})$ exists and is finite, and then the distribution of $S_n^*$ tends to  $\mathcal N(0,\sigma)$.
\vskip5mm

We start with recalling some properties of iterated function systems obtained by D. Malicet (see Theorem A and Corollary 2.6 in \cite{malicet}):  

\begin{prop}\label{Malicet}  Let $\Gamma=\{g_1,\ldots, g_k\}$ be a familiy of homeomorphisms on $\mathbf{S}^1$ such that there is no measure invariant by $\Gamma$. Let  ${\bf p}=(p_1,\dots p_k)$ be a probability vector. If $\Gamma$ acts minimally, then there exists $q\in (0, 1)$ such that:
\begin{itemize}
\item for every $x\in\mathbf{S}^1$ there exists an open neighbourhood $I$ of $x$ and $\tilde\Sigma\subset\Sigma$ with $\mathbb P(\tilde\Sigma)>0$ such that for ${\bf i}=(i_1, i_2, \ldots)\in\tilde\Sigma$ we have
$$
|g_{i_n,\ldots, i_1}(I)|\le q^n;
$$
\item (asymptotic stability) for any $x\in \mathbf{S}^1$ the sequence $(P^n\delta_x)_{n\ge 1}$, where $P$ is the Markov operator corresponding to $(\Gamma, p)$, converges in the $*$--weak topology to the unique invariant measure $\mu_*$.
\end{itemize}
\end{prop}
First, let us note that  Proposition ~\ref{Malicet} implies the e-property:
\begin{prop}\label{Stability_1} Under the hypothesis of Proposition \ref{Malicet} the operator $P$ corresponding to $(\Gamma, p)$ satisfies the e--property. Moreover, for any open interval $I\subset\mathbf S^1$ there exists $m\in\mathbb N$ such that
$$
\inf_{x\in\mathbf S^1} P^m\delta_x(I)>0.
$$
\end{prop}

\begin{proof} From Proposition \ref{Malicet} it follows that $P$ is asymptotically stable. Since $\supp\mu_*={\mathbf S}^1$, in particular $\text{Int}_{{\mathbf S}^1} \supp\mu_*\neq\emptyset$, from Theorem 4.8 in \cite{sander} we obtain that
$P$ satisfies the e--property.

Now fix an open interval $I\subset\mathbf S^1$. Since the operator $P$ corresponding to $(\Gamma, {\bf p})$ is asymptotically stable and it satisfies the e--property, for any $x\in{\mathbf S}^1$ there exists $N_x\in\mathbb N$ such that
\begin{equation}\label{e1_19.03.17}
P^n\delta_x(I)>\mu_*(I)/2>0 
\end{equation}
for all $n\ge N_x$. On the other hand, from the e--property it follows that for every $x\in {\mathbf S}^1$ we may choose some neighbourhood $U_x$ of $x$ such that the above property will be satisfied if we replace $x$ with an arbitrary $y$ from $U_x$. By compactness of ${\mathbf S}^1$ we may find $x_1, \ldots, x_k\in {\mathbf S}^1$ such that ${\mathbf S}^1=\bigcup_{i=1}^k U_{x_i}$. Then
for any $m\ge\max\{N_{x_1}, \ldots, N_{x_k}\}$ we have
$$
\inf_{x\in\mathbf S^1} P^m\delta_x(I)>\mu_*(I)/2>0
$$
and the proof is completed.
\end{proof}


We are going to introduce the following notation: for $\varphi: {\mathbf S}^{1}\to\mathbb R$, $x\in\mathbf S^1$, $\omega=(i_1,i_2,\dots )\in\Sigma$ or $\omega=(i_1,\dots i_r)\in \Sigma_r$, $r\ge n$, by  $S_n\varphi(\omega,x)$ we shall denote the sum:
$$
S_n\varphi(\omega,x)=\sum_{l=1}^{n}\varphi(g_{i_l,\dots, i_1}(x)).
$$

\begin{lem}\label{prop:pairing} Let $\Gamma=\{g_1,\ldots, g_k\}$ act minimally and there is no measure invariant by $\Gamma$. Assume that $p_1=\cdots=p_k=1/k$ and let $\varphi: {\mathbf S}^{1}\to\mathbb R$ be an arbitrary Lipschitz continous function with Lipschitz constant $1$.
Then there exist $\alpha\in (0, 1), m\in\mathbb N$ and $\gamma>0$ such that for an arbitrary pair of points   $x,y\in \mathbf S^1$ there exist two  collections of measurable pairwise  disjoint sets
$P_1,P_2\ldots\subset \Sigma$  and $P_1',P_2',\ldots\subset \Sigma$ such that: for every $l\ge 1$ 
$$\Sigma=P_1\cup P_2\cup\dots\cup P_l\cup B_l, \quad \quad \Sigma=P_1'\cup P_2'\cup\dots\cup P_l'\cup B_l',
$$ 
where the sets $B_l, B_l'$ are unions (possibly infinite) of some disjoint cylinders and a (measure preserving) bijection $b_l:\Sigma\to\Sigma$ satisfying
\begin{itemize}
\item[{\bf (P1)}] $b_l(P_j)=P_j'$ for every $j=1,\dots l$;
\item[{\bf (P2)}]${\mathbb P}(B_l)\le (1-\alpha)^l$;
\item[{\bf (P3)}] for every $n\ge 0$ and $\omega\in P_j$, $j=1,\ldots l$, we have
$$
|S_n\varphi(\omega,x)-S_n\varphi(b_l(\omega), y)|\le j\left (2(m+1)||\varphi||_\infty+\gamma\right);
$$
\item[{\bf (P4)}] for every cylinder set $C$ from the collection of cylinders forming $B_l$,  $n\le s$, where $s$ is the length of the cylinder $C$, we have
$$|S_n\varphi(\omega,x)-S_n\varphi(b_l(\omega), y)|\le j\left (2(m+1)||\varphi||_\infty+\gamma\right )\quad \omega\in C;
$$ 
\item[{\bf (P5)}] the bijection $b_{l+1}$ coincides with $b_l$ on the set $P_1\cup\dots\cup P_l$.
\end{itemize}
\end{lem}

\begin{proof} By Proposition \ref{Malicet} there exist an open interval $I\subset \mathbf {S}^1$, $q\in (0, 1)$ such that the set  $\tilde\Sigma$ of all sequences $\omega\in\Sigma$ satisfying the following condition:
\begin{equation}\label{pesin}
|g_{i_n, i_{n-1},\ldots, i_1}(z)-g_{i_n, i_{n-1},\ldots, i_1}(w)|\le q^n\quad\quad \text{for all~} n\ge 0, \quad\text{all}\quad z,w\in I
\end{equation} 
has positive $\mathbb P$- measure.
Further, from Proposition \ref{Stability_1} it follows that there exists $m\in\mathbb N$ such that
$$
\beta:=\inf_{x\in \mathbf {S}^1}P^m\delta_x(I)>0.
$$

Set
$$
\alpha:=\beta\mathbb P (\tilde\Sigma)\quad\text{and}\quad \gamma:=(1-q)^{-1}.
$$

Fix $x, y\in{\mathbf S}^1$. Then there exist a collection $\mathcal G\subset \Sigma_m$  of sequences  ${\bf i}=(i_1, \ldots, i_m)$ such that $g_{i_m, i_{m-1},\ldots, i_1}(x)\in I$, and, similarly, a collection $\mathcal G'$ of sequences ${\bf i'}=(i_1',\dots i_m')$ such that $g_{i'_m, i'_{m-1},\ldots, i'_1}(y)\in I$.  We may also assume that $G:=\#(\mathcal G)=\#(\mathcal G')>k^m\beta$. Fix an arbitrary bijection $${\bf i}\mapsto {\bf i'}$$ between sequences in $\mathcal G$ and $\mathcal G'$, respectively. 

Put  
$$
P_1=\{{\bf i}{\bf j}:{{\bf i}\in {\mathcal G},\bf j}\in \tilde \Sigma\}\quad \text{ and }\quad
P_1'=\{{\bf i'}{\bf j}:{{\bf i'}\in {\mathcal G}',  \bf j}\in \tilde \Sigma\}.
$$
Obviously $\mathbb P(P_1)\ge\alpha$.
The bijection $b_1:P_1\to P_1'$ is defined simply  by 
\begin{equation}\label{bijection}
b_1({\bf i}{\bf j})={\bf i'}{\bf j}\qquad\text{for every ${\bf j}\in \tilde \Sigma$.}
\end{equation}

Set $B_1=\Sigma\setminus P_1$ and observe that $\mathbb P(B_1)\le 1-\alpha$ and {\bf (P2)} is satisfied. Now we  define the bijection  $b_1$ on $B_1$. 
It is done in two steps: first, there are $k^m-G$  finite sequences $\bf l$  and ${\bf l'}$  of length $m$ outside $\mathcal G$ and $\mathcal G'$, respectively. Choose an arbitrary bijection  between them and define the bijection $b_1$ on the  
cylinder set defined by ${\bf l}$ by
\begin{equation}\label{bijection2}
b_1({\bf l}{\bf j})={\bf l'}{\bf j}
\end{equation}
for every ${\bf j}\in\Sigma$. Hence {\bf (P1)} holds.

Further from the definition of the set $\tilde\Sigma$ it follows that the complement $\Sigma\setminus\tilde\Sigma$ is a union, possibly infinite, of some disjoint cylinder sets, say
$$
\Sigma\setminus\tilde\Sigma=\bigcup_{{\bf k}\in K}C_{\bf k}.
$$
Also, let us note that for ${\bf k}\in K$ we still have the estimate
 \begin{equation}\label{pesin2}
|g_{ i_{n},\ldots, i_1}(z)-g_{ i_{n},\ldots, i_1}(w)|\le q^n \quad\text{for}\quad z,w\in I\,\,\text{and}\,\, n<r,
\end{equation}
where $r$ is the length of $\bf k$.

It remains to define $b_1$ on the complement $C_{\bf i}\setminus P_1$, for each sequence ${\bf i}\in \mathcal G$.  Since $\Sigma\setminus\tilde\Sigma$ is a union of some 
collection of  disjoint cylinders $C_{\bf k}$, ${\bf k}\in K$  in $\Sigma$, the set $C_{\bf i}\setminus P_1$ is the union 
of  cylinder sets $C_{\bf i k}$, ${\bf k}\in K$ and, similarly, the set $C_{\bf i'}\setminus P_1'$ is the union of the 
cylinders  $C_{\bf i' k}$, ${\bf k}\in K$. This defines a natural measure preserving  bijection $C_{\bf ik}\to C_{\bf 
i'k}$. Thus the definition of $b_1: \Sigma\to\Sigma$ is completed.
\vskip2mm
We shall check that {\bf (P3)}  is satisfied.
First, for $\omega\in P_1$, $\omega={\bf ij}$ and $\omega'=b_1(\omega)$ we have, for $n\ge m$:
\begin{equation}\label{eq:p1}
\aligned
&|S_n\varphi(\omega,x)-S_n\varphi(\omega',y)|=\\
&|S_m\varphi({\bf i},x)-S_m\varphi({\bf i'},y)
+S_{n-m}\varphi({\bf j},g_{\bf i}(x))-S_{n-m}\varphi({\bf j},g_{\bf i'}(y))|\\
&\le 2m||\varphi||_\infty+\gamma.
\endaligned
\end{equation}
If $n\le m$, then \eqref{eq:p1} holds trivially and {\bf (P3)} holds.

Now, let $C$ be a cylinder from the collection forming  $B_1$. If $C$ is of length $m$ then  {\bf (P4)} is trivially satisfied. Now,  if $C$ is of the form $\bf i\bf k$, so that the length of $C$ is equal to $s=m+r> m$, then {\bf (P4)} still holds for $n\le s$. Indeed,  applying \eqref{pesin2} for $r$, and $z=g_{\bf i}(x)$, $w=g_{\bf i'}(y)$, gives
$$|S_{n-m}\varphi({\bf j},g_{\bf i}(x))-S_{n-m}\varphi({\bf j},g_{\bf i'}(y))|
\le 2||\varphi||_\infty +\gamma$$
for $n\le s$ (i.e., $n-m\le r$), so that

$$|S_n\varphi(\omega, x)-S_n\varphi(\omega',y)|\le 2m||\varphi||_\infty+2||\varphi||_\infty+\gamma =2(m+1)||\varphi||_\infty+\gamma,
$$
where $\omega'=b_1(\omega)$.

First step in our induction argument is done.

\

Next, assume that hypotheses hold for $1,\dots l$. We shall construct the set $P_{l+1}\subset B_l$, and put $B_{l+1}=B_l\setminus P_{l+1}$.
The construction goes as follows: Let $C=C_{(i_1,\dots i_r)}$ be a cylinder set form the collection forming $B_l$, and let $C'=b_l(C)=C_{(i_1',\dots i_r')}$.
Again, there exist collections, both of cardinality $G$   of  sequences of length m $(i_{r+1},\ldots, i_{r+m})$, $(i'_{r+1},\ldots, i'_{r+m})$ such that both 
$g_{i_{r+m},\ldots, i_{r+1}, i_r,\ldots,  i_1}(x)$
 and $g_{i'_{r+m},\ldots, i'_{r+1}, i'_r,\ldots,  i'_1}(y)$ are in $I$.  Choose an arbitrary bijection between them $(i_{r+1},\ldots, i_{r+m})\mapsto (i'_{r+1},\ldots, i'_{r+m})$. 
Put ${\bf i}=(i_1,\ldots,
i_{r+m})$ and ${\bf i'}=(i_1',\ldots, i_{r+m}')$ and consider the subsets of the  cylinder sets defined by $\bf i$ and $\bf 
i'$:
$$
\{{\bf i j}:{\bf j}\in\tilde\Sigma\}\quad\quad\text {and}\quad\quad  \{{\bf i' j}:{\bf j}\in\tilde\Sigma\},
$$ and the natural bijection $b_{l+1}$ between them: $\bf i\bf j\mapsto\bf i'\bf j$. Let $b_{l+1}(\omega)=b_{l}(\omega)$ for all $\omega\in\Sigma\setminus P_{l+1}$. Thus {\bf (P5)} holds. Obviously {\bf (P1)} is also satisfied.

 The set $P_{l+1}$ (respectively: $P_{l+1}'$) is   then defined as the union of all such sets  $(\bf i \bf j)$ constructed above, over all cylinder sets in $B_l$. 
  
 It follows from the structure of $\tilde\Sigma$  that $B_{l+1}=B_l\setminus P_{l+1}$ is, again , a union of some 
 cylinder sets. The construction gives also the estimate ${\mathbb P}(P_{l+1})\ge\alpha {\mathbb P}(B_l)$, so 
 ${\mathbb P}(B_{l+1})\le (1-\alpha)^{l+1}$ and {\bf (P2)} holds.
 
 Condition {\bf (P3)} now holds for $P_{l+1}$. Indeed, take one of cylinder sets   $C=C_{(i_1,\dots i_r)}$  forming the set $B_l$, and follow the above construction, i.e. extend the sequence to ${\bf i}=(i_1,i_2,\dots i_r,i_{r+1},\dots i_{r+m})$ and repeat the same procedure  for for $C'$. Take $\omega\in P_{l+1}$, $\omega={\bf i}{\bf j}$ and $\omega'=b_{l+1}(\omega)={\bf i'}{\bf j}$.
 Then, by inductive assumption, 
 $$|S_{r+m}({\bf i},x)-S_{r+m}({\bf i'}, y)|\le l(2(m+1)||\varphi||_\infty+\gamma)+2m||\varphi||_\infty.\,$$
 and for every $n>r+m$
 $$|S_n({\bf j},g_{\bf i}(x))-S_n({\bf j},g_{{\bf i'}}(y))|\le \gamma.$$
 Summing these two estimates we obtain {\bf (P3)} for $P_{l+1}$. Similarly we check that {\bf (P4)} holds. The proof is complete.
 \end{proof}
 
We may formulate the main result of our paper saying that the iterated function system under quite general assumptions fulfils the Central Limit Theorem.

\begin{thm}\label{thm:clt} Let $\varphi: \mathbf S^1\to\mathbb R$ be an arbitrary Lipschitz continuous function. If $\Gamma=\{g_1,\ldots, g_k\}$ acts minimally and there is no measure invariant by $\Gamma$, then for any probability vector ${\bf p}=(p_1,\dots p_k)$ the iterated function system $(\Gamma, {\bf p})$ satisfies the Central Limit Theorem for the function $\varphi$.
\end{thm}

\begin{proof} First we assume   that $\varphi$ is Lipschitz continuous; one can also assume that  the Lipschitz  constant of the function $\varphi$ is equal to $1$.  
\vskip2mm
At the first step of the proof we shall assume that all the probabilities $p_1,\dots p_k$ are equal, i.e., 
$p:=p_1=p_2=\dots =p_k=\frac{1}{k}.$ 
The general case will be deduced from this special case at the end of the proof.

\

Fix $n\in\mathbb N$. Observe that

\begin{equation}
\sum_{k=1}^nU^k\varphi(x)=p^n\sum_{\omega\in\Sigma_n}S_n\varphi(\omega,x)
\end{equation}
 
So, for $x,y\in \mathbb S^1$ we have
\begin{equation}\label{eq:diff}
\aligned
&|\sum_{k=1}^nU^k\varphi(x)-\sum_{k=1}^nU^k\varphi(y)|=|p^n\sum_{\omega\in\Sigma_n}S_n\varphi(\omega,x)-p^n\sum_{\omega\in\Sigma_n}S_n\varphi(\omega,y)|\\
\endaligned
\end{equation}

Let $\beta\in (0,1)$ and let $\Sigma=P_1\cup P_2\cup\dots \cup P_{[n^\beta]}\cup B_{[n^\beta]}$, where $P_1, P_2, \ldots, P_{[n^\beta]}$ and $B_{[n^\beta]}$ are given by Lemma \ref{prop:pairing}. Then

$$\sum_{\omega\in\Sigma}S_n\varphi(\omega, x)=\sum_{\omega\in P_1\cup\dots\cup P_{[n^\beta]}}S_n\varphi(\omega, x)+\sum_{\omega\in B_{[n^\beta]}}S_n\varphi(\omega, x)$$ 
and, similarly, we also have

$$
\sum_{\omega\in\Sigma}S_n\varphi(\omega, y)=\sum_{\omega\in P'_1\cup\dots\cup P'_{[n^\beta]}}S_n\varphi(\omega, y)+\sum_{\omega\in B'_{[n^\beta]}}S_n\varphi(\omega, y),
$$ 
where $P_i'=b_{[n^\beta]} (P_i)$ for $i=1, \ldots, [n^\beta]$ and $B'_{[n^\beta]}=b_{[n^\beta]}(B_{[n^\beta]})$.

We need to estimate  $p^n\left (\sum_{\omega\in\Sigma_n}(S_n\varphi(\omega,x)-S_n\varphi(\omega,y))\right )$. Using the bijecton $b_{[n^\beta]}$ and defining $b_{[n^\beta]}(\omega)=\omega'$  we have
$$
\aligned
&p^n\left (\sum_{\omega\in\Sigma_n}(S_n\varphi(\omega,x)-S_n\varphi(\omega,y))\right )
=p^n\left(\sum_{\omega\in P_1\cup\dots\cup P_{[n^\beta]}}S_n\varphi(\omega, x)-\sum_{\omega'\in P_1'\cup\dots\cup P'_{[n^\beta]}}S_n\varphi(\omega', y)\right)\\
&+p^n\sum_{\omega\in B_{[n^\beta}]}(S_n\varphi(x,\omega)-S_n\varphi(y,\omega'))
:=I+II.
\endaligned 
$$
By {\bf (P3)} and {\bf (P4)} in Lemma \ref{prop:pairing} we  can estimate the above summands:
$$|I|\le \mathbb P(P_1\cup\dots\cup P_{[n^\beta]})n^\beta (2(m+1)||\varphi||_\infty+\gamma)\le n^\beta(2(m+1)||\varphi||_\infty+\gamma),$$
$$|II|\le 2n||\varphi||_\infty\cdot \mathbb P(B_{[n^\beta]})\le 2n||\varphi||_\infty\cdot (1-\alpha)^{n^\beta}.
$$ 
Summarizing, we obtain the following estimate:
\begin{equation}\label{eq:diff2}
|\sum_{k=1}^n (U^k\varphi(x)-U^k\varphi(y))|\le C n^\beta,
\end{equation}
where $C$ is some constant depending on $m,\gamma, \beta, \alpha$ and $||\varphi||$. Therefore,
$$
\begin{aligned}
|\sum_{k=1}^n U^k\varphi(x)|&=|\sum_{k=1}^n U^k\varphi(x)-\int_{\mathbf{S}^1}\varphi(y)\mu_*(\d y)|=|\int_{\mathbf{S}^1}[\sum_{k=1}^n U^k\varphi(x)-U^k\varphi(y)]\mu_*(\d y)|\\
&\le  \int_{\mathbf{S}^1}|\sum_{k=1}^n (U^k\varphi(x)-U^k\varphi(y))|\mu_*(\d y)\le Cn^\beta.
\end{aligned}
$$
Clearly, this uniform estimate implies that 
$$||\sum_{k=1}^n U^k\varphi||_{L_2(\mu_*)}\le C n^\beta:=a_n.
$$
The above estimate can be performed for every $\beta\in (0,1)$. Choosing some $\beta\in (0,1/2)$, e.g., $\beta=1/4)$,
we see that the series
$
\sum_{n=1}^{\infty}\frac{a_n}{n^{3/2}}
$
is convergent. Thus, condition \eqref{eq:1_6/03.17} holds and the stationary sequence $(\varphi(X_n))_{n\ge 1}$ satisfies the CLT.

To show that the CLT theorem holds for a sequence $(\varphi(X^x_n))_{n\ge 1}$ starting at arbitrary 
$x\in\mathbf{S}^1$ it is enough to prove that 
\begin{equation}\label{eq:pointwise}
\left|\mathbb E \exp\left(it \frac{\varphi(X_1^x)+\ldots+\varphi(X_n^x)}{\sqrt{n}}\right)
-\mathbb E \exp\left(it \frac{\varphi(X_1)+\ldots+\varphi(X_n)}{\sqrt{n}}\right)\right|\to 0\quad\text{as $n\to\infty$.}
\end{equation}
or, in our notation, that the following difference

$$
\aligned
&\int _\Sigma\exp \left(it\frac{S_n\varphi(\omega,x)}{\sqrt n}\right)\mathbb P
(\d\omega)-\int_\Sigma
\int_{\mathbb S ^1}\exp \left(it 
\frac{S_n\varphi(\omega,y)}{\sqrt n}\right)\mu_*(\d y)\mathbb P(\d\omega)\\
&=\int _\Sigma\int_{\mathbb S^1}\left (\exp \left(it\frac{S_n\varphi(\omega,x)}{\sqrt n}\right)- \exp \left(it 
\frac{S_n\varphi(\omega,y)}{\sqrt n}\right)\right )\mu_*(\d y)\mathbb P(\d\omega)
\endaligned
$$ 
converges to $0$ as $n\to\infty$.

With $x$ and $y$ fixed, the expression in the  brackets can be estimated
by 
$$
\frac{p^n}{\sqrt n}\left|\sum_{\omega\in\Sigma_n}S_n\varphi(\omega,x)-\sum_{\omega\in\Sigma_n}S_n\varphi(\omega,y)\right|=\frac{1}{\sqrt n}\left|\sum_{k=1}^n(U^k\varphi(x)-U^k\varphi(y))\right|.
$$
Using \eqref{eq:diff2}, again for $\beta=1/4$, we conclude  that \eqref{eq:pointwise} holds.

Since 
$$
\mathbb E \exp\left(it \frac{\varphi(X_1)+\ldots+\varphi(X_n)}{\sqrt{n}}\right)\to \exp (-t^2\sigma^2/2)\qquad\text{ as $n\to\infty$},
$$ 
we obtain 
$$\mathbb E \exp\left(it \frac{\varphi(X_1^x)+\ldots+\varphi(X_n^x)}{\sqrt{n}}\right)\to\exp (-t^2\sigma^2/2)\qquad\text{ as $n\to\infty$}
$$
and we are done.

\

Now, assume that an arbitrary probability vector  ${\bf p}=(p_1,\dots p_k)$ is given. We shall deduce the CLT for this general case from the previous case of equal probabilities.

First, assume that all $p_1,\dots p_k$  are rational; say $p_i=\frac{m_i}{n}$, $i=1,\dots k$.
Consider a modified symbolic space $\hat\Sigma$: this is the space of infinite sequences built with $N:=m_1+m_2+\dots m_k$ digits: 
$$
1^{(1)}, \dots, 1^{(m_1)}; 2^{(1)},\dots, 2^{(m_2)};\dots k^{(1)},\dots, k^{(m_k)}.
$$
Assigning equal probabilities ($=\frac{1}{N}$) to each digit, we obtain a new probability space
$(\hat\Sigma, \hat{\mathbb{P}})$, and a new (formally) IFS, assigning to each digit $i^{(t)}$, $t=1, \dots, m_i$ the same map $g_i$ for $i=1, \ldots, k$.  Denote by $\hat U$ the operator corresponding to this new IFS.
 Note that the natural projection $\Pi:(\hat\Sigma, \hat{\mathbb P})\to (\Sigma,\mathbb P)$ is measure preserving, i.e. $\hat{\mathbb P}(\Pi^{-1}A)=\mathbb P(A)$ for every measurable set $A$.  Thus, the systems $\Gamma$ and $\tilde\Gamma$ share the same  stationary measure $\mu_*$, and
 $$U^{k}\varphi=\hat U^k\varphi.$$
Therefore, estimate  \eqref{eq:diff2} implies that the identical estimate  holds unchanged for the system $(\Gamma, \bf p)$.
 Since this is all what we need to conclude CLT, we are done for this (rational) case.
 
 Fixing say, $\beta=1/4$,  recall that 
the constant $C$ depends on $||\varphi||$, and on  the constants  $m, \gamma, \alpha$, where $m$ comes from \eqref{e1_19.03.17}, $\gamma=(1-q)^{-1}$, where $q$ appears in the definition of $\tilde\Sigma$, see \eqref{pesin}  and  $\alpha=\beta{\mathbb P}(\tilde\Sigma)$, where $\beta=\mu_*(I)/2$.
\vskip2mm
Finally,  let ${\bf p}=(p_1,\dots p_k)$ be an arbitrary probability vector, let $\mu_*$  be the unique invariant measure for this system. Fix $m$ satisfying 
\eqref{e1_19.03.17}.
Choose the set $\tilde\Sigma$,  as in  \eqref{pesin}, and the constant $\delta$ coming from the definition of $\tilde\Sigma$. Put $\beta=\mu_*(I)/2$ and $\alpha=\beta{\mathbb P}(\tilde\Sigma)$, as before.

Now, choose and fix some  $n\ge m$. 
Note that if a rational  probability vector $\hat{\bf p}$, generating the product probability 
distribution $\hat{\mathbb P}$ on $\Sigma$  is close to $\bf p$  then \eqref{e1_19.03.17} still 
holds for the modified system, with the same $m$. Similarly, if $\hat{\bf p}$ is close to ${\bf 
p}$
 and $\hat{\mu}_*$ 
 is the corresponding stationary measure then $\hat\mu_*(I)$ is close to $
\mu_*(I)$.

The  estimates leading to \eqref{eq:diff2} for this rational approximation depend also, formally,  on lower estimate of $\hat{\mathbb P}(\tilde\Sigma)$, the probability which may change  after this approximation. However,  it is easy to see that, with this fixed $n$, the only lower  bound which is used to obtain condition \eqref{eq:diff2} is that of $\hat{\mathbb P}_n(\tilde\Sigma_n)$, where $\tilde\Sigma_n$ is the union of all cylinder sets of length $n$ which intersect $\tilde\Sigma$.
Clearly, given $n$, one can find a rational approximation of $\hat {\bf p}$ so that 
 $\hat{\mathbb P}_n(\tilde\Sigma_n)$ is as close to ${\mathbb P}_n(\tilde\Sigma_n)$ as we wish. Thus, \eqref{eq:diff}, and, in consequence, CLT holds for $(\Gamma, {\bf p})$. The proof is complete. 
\end{proof}

\begin{rem} The same theorem holds for a H\"older continuous observable $\varphi$. The above proof goes through with obvious modifications.
\end{rem}

 \end{document}